\documentclass[10pt, a4paper]{article}
\usepackage[T2A]{fontenc}
\usepackage[english]{babel}
\usepackage{amssymb,amsmath,amsfonts,amsthm,color}
\usepackage{graphicx, srcltx, dsfont}

\usepackage[vmargin={3.0cm,2.0cm},hmargin={3.0cm,2.0cm},includefoot]{geometry}

\theoremstyle{plain}
\newtheorem{theorem}{Theorem}
\newtheorem{lemma}{Lemma}

\theoremstyle{definition}

\numberwithin{equation}{section}

\def\Om{\Omega}

\def\g{\gamma}
\def\G{\Gamma}
\def\l{\lambda}
\def\p{\partial}

\def\a{\alpha}
\def\b{\beta}

\def\L{\Lambda}

\def\hf{\mathfrak{h}}

\def\Op{\mathcal{H}}

\def\cR{\mathcal{R}}

\def\cE{\mathcal{E}}

\def\cB{\mathcal{B}}
\def\cL{\mathcal{L}}
\def\cT{\mathcal{T}}

\def\cI{\mathcal{I}}

\DeclareMathOperator{\RE}{Re}
\DeclareMathOperator{\IM}{Im}

\renewcommand{\geq}{\geqslant}

\allowdisplaybreaks

\begin{document}

\begin{center}
\textbf{Spectral properties of Schr\"odinger operator with translations and Neumann
boundary conditions}
\end{center}
\begin{center}
D.I.~Borisov$^{1,2},$
D.M.~Polyakov$^{1,3}$
\end{center}
\begin{quote}
1) Institute of Mathematics, Ufa Federal Research Center, RAS,
\\
Ufa, Chernyshevsky
str. 112, 450008, Russia,
\\
2) Peoples Friendship University of Russia (RUDN University),
\\
 Moscow, 117198, Miklukho--Maklaya str. 6,  Russia,
\\
3) Southern Mathematical Institute, Vladikavkaz Scientific Center of RAS,
\\
Vladikavkaz,
362025, Vatutin str. 53, Russia.
\\
Emails:  borisovdi@yandex.ru, dmitrypolyakow@mail.ru.
\end{quote}
\begin{abstract}
We consider a nonlocal differential--difference Schr\"odinger operator on a segment
with the Neumann conditions and two translations in the free term. The values of the translations are denoted by $\alpha$ and $\beta$ and are treated as parameters.  The spectrum of this operator consists of countably many discrete eigenvalues, which  are taken in the ascending order of their absolute values and are indexed by the natural parameter $n.$ Our main result is the representation of the eigenvalues as convergent series in negative powers of $n$ with the coefficients depending on $n,$ $\a,$ and $\b.$ We show that these series converge absolutely and uniformly in $n,$ $\a,$ and $\b$
and they can be also treated as spectral asymptotics for the considered operator with uniform in $\a$ and $\b$ estimates for the error terms. As an example,  we find the four--term spectral asymptotics for the eigenvalues   with the error term of order
$O(n^{-3}).$ This asymptotics involves additional nonstandard
terms and exhibits a non--trivial high--frequency phenomenon generated by the translations.
We also establish that the system of eigenfunctions and generalized eigenfunctions
of the considered operator forms the Bari basis in the space of functions square integrable on the unit
segment.

\bigskip

\end{abstract}

\normalsize
\sloppy
\vspace{5pt}

\section{Introduction}

Differential--difference equations involve a differentiation operator  and  a translation operator. 
The solutions to these equations exhibit interesting non--classical properties generated by the presence of the translation operators. 
There are many works devoted to studying differential--difference equations.
Most of them develop qualitative theory and establish results on well--posedness of boundary and initial
problems for differential--difference equations, solvability issues, and the
behavior of solutions, see the book \cite{Skub_Book}, the papers
\cite{SkubUMN}, \cite{NevSkub}, \cite{SkubMathNach}, \cite{LiikoSkub},
\cite{Muravnik11}, \cite{Kamenskii} and the references therein.
Necessary and sufficient conditions for preserving the smoothness of
generalized eigenfunctions over the entire interval were obtained in \cite{VorotSkub}.
Evolutionary differential--difference equations were investigated quite intensively (see,
for instance, \cite{Muravnik1}, \cite{Muravnik2}, \cite{Muravnik3}, \cite{Muravnik4},
\cite{Muravnik5}).

Similar nonlocalities in differential equations can be produced by the
  dilatation and contraction operators. The properties of solutions to various
problems for such equations were studied in \cite{Rosov1}, \cite{Rosov2},
\cite{Rosov3}, \cite{Rosov4}.

The spectral properties of differential--difference operators
were investigated in \cite{SkubUMN} and the main results  were the completeness
and minimality of the system of eigenfunctions and generalized eigenfunctions
for elliptic differential--difference operators.

Inverse problems for the Sturm~---~Liouville operator with a constant translation and two--point
boundary conditions were considered in \cite{FreilingYurko}, \cite{Yang},
\cite{Bond_Yurko}, \cite{DjuricVladicic}, \cite{WangShiehMiao}, \cite{DjuricButerin1},
\cite{DjuricButerin2}. Some aspects of inverse problems for Sturm~---~Liouville type
operators with several delays were studied in \cite{Shahriari1}, \cite{Shahriari2},
\cite{VojvodicPikula}. Necessary and sufficient conditions for solvability of
functional--differential pencils possessing a general form along with a nonlinear
dependence on the spectral parameter were considered in \cite{ButerinMalShieh}.

A resolvent convergence for general operators with small variable translations were considered
in \cite{Borisov_Polyakov}. In this paper we studied elliptic operators of even order $2m$
in arbitrary domains with general boundary conditions. The operators involved small variable translations in lower order terms and boundary conditions. The main result of this work established the uniform resolvent
convergence in the norm of operators acting from $L^2$ to $W_2^m$ and estimates for the convergence rate. The convergence of the spectra and pseudospectra of perturbed operators to similar sets of limit operators was shown.

In \cite{Borisov_PolyakovDAN} we considered the Schr\"odinger operator on a segment with the Dirichlet condition, which involved a translation operator in the free term; the translation parameter was supposed to be small. The main result was the uniform spectral asymptotics, that is, the asymptotics for the eigenvalues in their index with an error term uniform in the translation parameter. Similar results were obtained in \cite{Borisov_PolyakovUMJ} for a model problem with the Dirichlet condition at one end--point and the Neumann condition at the other. The differential expression in \cite{Borisov_PolyakovUMJ} had constant coefficients and involved a small translation. The case of Neumann condition was addressed in \cite{Borisov_PolyakovIZV}, and the periodic boundary conditions were treated in \cite{RJMP2025}.

In the present paper we consider  a nonlocal differential--difference
Schr\"odinger operator on a segment with the Neumann condition, which involves two
translations in the free term. The translations are arbitrary and are not supposed to be small; we denote the values of translations by $\a$ and $\b.$ The spectrum of this operator consists of countably many discrete eigenvalues, which  are taken in the ascending order of their absolute values and are indexed by the natural parameter $n.$ Our main result provides   the representation of the eigenvalues as convergent series in negative powers of $n$ with the coefficients depending on $n,$ $\a,$ and $\b.$ We show that these series converge absolutely and uniformly in $n,$ $\a,$ and $\b$
and they can be also treated as spectral asymptotics for the considered operator with uniform in $\a$ and $\b$ estimates for the error terms. As an example,  we find the four--term spectral asymptotics for the eigenvalues with the error term of order
$O(n^{-3}).$ This asymptotics involves additional non--standard
terms and exhibits a non--trivial high--frequency phenomenon generated by the translations. We also establish that the system of eigenfunctions and generalized eigenfunctions of the considered operator forms the Bari basis in the space of functions square integrable on the unit segment.

To the best of our knowledge, earlier uniform spectral properties were obtained for the problem
\begin{equation}\label{0.1}
i\varepsilon y''+q(x)y=\lambda y, \qquad \qquad y(a)=y(b)=0,
\end{equation}
where $\lambda$ is a spectral parameter and $\varepsilon$ is a small parameter.
The pseudospectra of \eqref{0.1} was investigated in \cite{Reddy} and \cite{Trefethen}.
Uniform spectral asymptotics for the problem \eqref{0.1} in the case $q(x)=x$
and $q(x)=x^2$ were found in \cite{ShkalikovNotes97} and \cite{TumanovShkalikov02}, see also \cite{Stepin1}, \cite{Stepin2}, \cite{Morawetz},
\cite{DyachenkoShkalikov}, \cite{Shkalikov03},  \cite{Chapman}.
In this situation the spectrum is concentrated near a line consisting of three curves and resembling a tie.  In \cite{PokotiloShkalikov09} and \cite{IshkinMarv} these results were
extended to the situation when $q$ belongs to the class of analytical functions.
The case of periodic boundary conditions was considered in \cite{EsinaShafar1},
\cite{EsinaShafar2}. Note that if the positive parameter $\varepsilon$ is involved in
the differential expression instead of $i\varepsilon$, then we obtain a self--adjoint
problem with a small parameter. This is the well--known classical problem on quasiclassical approximation for the Schr\"odinger operator and here the spectrum is described by Bohr~---~Sommerfeld quantization, see  \cite{Shubin}. We also mention our recent work \cite{RJMP2024}, in which uniform spectral asymptotics were found for the Schr\"odinger operator on a segment with the Dirichlet condition, which was perturbed by a delta--interaction located at a point near one of the ends of segment.

The results of this paper were partially announced in \cite{LJM2024}.

\section{Problems and main results}

In the space $L^2(0, 1)$ we consider the self--adjoint operator  $\mathcal{A}y=-y''$
on the domain $$\mathfrak{D}(\mathcal{A})=\big\{y\in W_2^2(0, 1): y'(0)=y'(1)=0\big\}.$$   By
$\cE $ we denote the  operator of continuation by zero outside $(0, 1)$, while
$\mathcal{R}$ stands for the operator of restriction to $(0, 1)$.
The operator $\cE $ acts from $L^2(0, 1)$ into $L^2(\mathds{R})$ and
continues the function by zero outside $(0, 1)$, while the operator $\mathcal{R}$
acts $L^2(\mathds{R})$ into $L^2(0, 1)$ and makes the restriction:
\[
\cE y=
\begin{cases}
y \quad \text{in} \quad (0, 1), \\
0 \quad \text{outside} \quad (0, 1),
\end{cases}
\qquad \mathcal{R}y=y \quad \text{on} \quad (0, 1).
\]
We define the translation operator  $\mathcal{T}$ 
in $L^2(\mathds{R})$:
\[
(\mathcal{T}(\alpha) y)(x)=y(x+\alpha), \qquad  x\in (0, 1),
\]
where $\alpha\in [-1, 1]$ is a parameter.

Let $V$ and $Q$ be complex--valued functions belonging to the space $L_\infty (0, 1)$.
We introduce the operator $\mathcal{B}(\alpha, \beta)$ acting into $L^2(0, 1)$ by the rule
\begin{equation}\label{B}
\mathcal{B}(\alpha, \beta) =V\mathcal{R}\mathcal{T}(\alpha)\cE
+Q\mathcal{R}\mathcal{T}(-\beta)\cE,
\end{equation}
where $\alpha, \beta\in [0, 1]$.

The main object of this study is the nonlocal operator $\mathcal{H}(\alpha, \beta)
=\mathcal{A}+\mathcal{B}(\alpha, \beta)$
in $L^2(0, 1)$ on the domain $\mathfrak{D}(\mathcal{H}(\alpha, \beta)):=\mathfrak{D}(\mathcal{A})$.
For an arbitrary function $y\in\mathfrak{D}(\mathcal{H}(\alpha, \beta))$, the action of
$\mathcal{H}(\alpha, \beta)$ on $y$ reads as
\[
(\mathcal{H}(\alpha, \beta)y)(x)=-y''(x)+V(x)y(x+\alpha)+Q(x)y(x-\beta),
\]
where the function $y$ is supposed to be continued by zero outside the segment $[0, 1]$
and after that the translation it is restricted to the segment $[0, 1]$. For $\alpha=\beta=0$  this operator becomes the classical Schr\"odinger operator
\[
\mathcal{H}(0, 0)y=-y''+(V+Q)y
\]
in $L^2(0, 1)$ on the domain $\mathfrak{D}(\mathcal{H}(0, 0))=\mathfrak{D}(\mathcal{A})$. If $\alpha=\beta=1$, then $\mathcal{H}(1, 1)=\mathcal{A}$.

The aim of this paper is to study the behavior of the spectrum of   operator
$\mathcal{H}(\alpha, \beta)$ as $\alpha$, $\beta$ vary in $[0, 1]$.

 By $\|\,\cdot\,\|_{X\to Y}$ we denote the norm of a bounded operator acting from a Banach space $X$ into a Banach space $Y.$

Our first result describes the main properties of the spectrum
of the operator $\mathcal{H}(\alpha, \beta)$.

\begin{theorem}\label{th0}
The operator $\mathcal{H}(\alpha, \beta)$ is $m$--sectorial and the associated sesquilinear
form in the space $L^2(0, 1)$ is
\begin{equation}\label{2.1}
\mathfrak{h}(\alpha, \beta)[u, v]=(u', v')_{L^2(0, 1)}
+(V\mathcal{R}\mathcal{T}(\alpha)\cE u, v)_{L^2(0, 1)}
+(Q\mathcal{R}\mathcal{T}(-\beta)\cE u, v)_{L^2(0, 1)}
\end{equation}
on the domain $\mathfrak{D}(\mathfrak{h}(\alpha, \beta)):=\{u\in W_2^1(0, 1)\}$.
There exists a fixed number $\Lambda$ independent of $\alpha$ and $\beta$ such
that the half-plane $\{\lambda\in\mathds{C}: \RE \lambda\leqslant \Lambda\}$ is in the resolvent  set of the operator $\mathcal{H}(\alpha, \beta)$ for all $\alpha$, $\beta\in [0, 1]$. For each $\l\in\mathds{C},$ $\RE\l\leqslant \L$ the estimates
\begin{equation}\label{2.3}
\begin{aligned}
&\|(\Op(\a,\b)-\l)^{-1}\|_{L_2(0,1)\to W_2^2(0,1)}\leqslant C,
\\
&\|(\Op(\a_1,\b_1)-\l)^{-1}-(\Op(\a_2,\b_2)-\l)^{-1}\|_{L_2(0,1)\to W_2^2(0,1)}\leqslant C\big(|\a_1-\a_2|^\frac{1}{2}+|\b_1-\b_2|^\frac{1}{2}\big)
\end{aligned}
\end{equation}
hold, where $C$ are some constants independent of $\a,\,\b,\,\a_1,\,\b_1,\,\a_2,\,\b_2\in[0,1].$
The operator $\mathcal{H}(\alpha, \beta)$ has a compact resolvent and its spectrum consists of a countably many eigenvalues with the only accumulation point at infinity. These eigenvalues are continuous in $\a$ and $\b.$
\end{theorem}

Let $\lambda_n=\l_n(\a,\b)$, $n\in\mathds{N}$, be the eigenvalues of the operator $\mathcal{H}(\alpha, \beta)$. We enumerate
these eigenvalues in the ascending order of their absolute values.

Now we formulate the main result of our maniscript. In the space $C[0, 1]$ we introduce the linear operators
\begin{equation}\label{OperM}
\begin{aligned}
&\big(\mathcal{M}(z,\alpha, \beta)u\big)(x, z):= \int\limits_0^x(\mathcal{B}(\alpha, \beta)u)(t, z)\sin z(x-t)\,dt, \\
&\big(\cL(z,\alpha, \beta)u\big)(x, z):=
\int\limits_0^{x}(\mathcal{B}(\alpha, \beta)u)(t, z)\cos z(x-t)\,dt,
\end{aligned}
\end{equation}
and the functions
\begin{equation}\label{fj}
f_j(z,\alpha, \beta):=\Big(\cL(z,\alpha, \beta)(\mathcal{M}(z,\alpha, \beta))^j\Phi(\,\cdot\,,z)\Big)\big|_{x=1},
\quad z\in\mathds{C},\quad  j\geq 0, \qquad \Phi(x,z):=\cos zx.
\end{equation}
We denote the $m$th derivative of $f_j$ with respect to  $z$ by $\p_z^m f_j.$

Our second main result provides the representations for the eigenvalues $\l_n$ in terms of converging series  in inverse powers of $n$.

\begin{theorem}\label{th1}
Let $V$ and $Q$ be complex--valued functions belonging the space $L_\infty (0, 1)$. There exists a fixed $n_0>0$ independent of $\a,\,\b\in[0,1]$ such that for $n\geqslant n_0$ the eigenvalues $\lambda_n$ of the operator $\mathcal{H}(\alpha, \beta)$ are simple and can be represented as
\begin{equation}\label{prom11z}
\lambda_n(\a,\b)=\Bigg(\pi n+\frac{1}{\pi n}\sum\limits_{i=0}^\infty\frac{\rho_i(n,\alpha, \beta)}{(\pi n)^i}\Bigg)^2.
\end{equation}
The series in the above representation converges uniformly in $\alpha,\, \beta\in[0,1]$ and $n\geqslant n_0.$ The coefficients
$\rho_i(n,\alpha, \beta)$  are given by the formulas
\begin{equation}\label{polDi}
\begin{aligned}
\rho_0(n,\alpha, \beta) &= (-1)^n\omega_0(n,\alpha, \beta), \quad
\rho_1(n,\alpha, \beta)=(-1)^n\omega_1(n,\alpha, \beta), \\
\rho_i(n,\alpha, \beta) &= (-1)^n\omega_i(n,\alpha, \beta)-\sum\limits_{s=1}^{[i/2]} H_{s, i-2s}(n,\a,\b),
\qquad i\geqslant 2,
\end{aligned}
\end{equation}
where $[\cdot]$ is the integer part of a number,
\begin{align}\label{Gp}
&\omega_i(n,\alpha, \beta):=\sum\limits_{p=0}^i\sum\limits_{t=0}^{i-p}E_{t, i+1-t-p}(n,\a,\b)
\sum\limits_{\substack{\zeta\in\mathds{Z}_+^t \\ |\zeta|=p}}
\rho_{\zeta_1}(n,\a,\b)\cdot\dots\cdot\rho_{\zeta_t}(n,\a,\b),
\qquad i\geqslant 0,
\\
&\label{Etl}
E_{t, l}(n,\a,\b):=\sum\limits_{j=\max\{0, l-t-1\}}^{l-1}\frac{(-1)^{l-j-1}(l-1)!
 \p_z^{(t-l+j+1)}f_j(\pi n,\a,\b)}{(l-j-1)!j!(t-l+j+1)!}, \quad t\geqslant 0, \quad l\geqslant 1,
\\
&\label{Hsm}
H_{s, m}(n,\a,\b):=\frac{(-1)^s}{(2s+1)!}\sum\limits_{\substack{\nu\in\mathds{Z}_+^{2s+1} \\ |\nu|=m}}
\rho_{\nu_1}(n,\a,\b)\cdot\ldots\cdot\rho_{\nu_{2s+1}}(n,\a,\b),
\quad s\geqslant 1, \quad m\geqslant 0,
\end{align}
and we adopt
\begin{equation}
\rho_{\zeta_1}(n,\a,\b)\cdot\ldots\cdot\rho_{\zeta_t}(n,\a,\b):=
\left\{
\begin{aligned}
& 1 \quad\text{for}\quad t=0, \quad p=0,
\\
&0 \quad \text{for}\quad t=0,\quad p>0.
\end{aligned}
\right. \label{2.2}
\end{equation}
The estimates
\begin{equation}\label{4.3}
|\rho_i(n,\a,\b)|\leqslant c_{1}c_{2}^i
\end{equation}
hold, where $c_{1},$ $c_{2}$ are some constants independent of $i,$ $n,$ $\a,$ $\b.$ For each $N\geqslant 1$ the remainder of the series in (\ref{prom11z}) satisfies the bound
\begin{equation}\label{4.19}
\bigg|\sum\limits_{i=N}^{\infty} \frac{\rho_i(n,\alpha, \beta)}{(\pi n)^i}\bigg| \leqslant \frac{2c_{1}c_{2}^N}{\pi^N n^N}\quad\text{for}\quad n\geqslant n_0.
\end{equation}
\end{theorem}

To demonstrate how the recurrent formulas \eqref{polDi} work, we have calculated the first three coefficients by these formulas:
\begin{equation}\label{coeffD}
\begin{aligned}
\rho_0(n,\alpha, \beta)=&(-1)^nf_0(\pi n,\a,\b),
\\
\rho_1(n,\alpha, \beta)=&f_0(\pi n,\a,\b)\p_z f_0(\pi n,\a,\b)
+(-1)^n f_1(\pi n,\a,\b), \\
\rho_2(n,\a,\b)=& \frac{(-1)^nf_0^3(\pi n,\a, \b)}{6}
+(-1)^nf_0(\pi n,\a,\b)
(\p_z f_0(\pi n,\a,\b))^2
\\
&+\p_z f_0(\pi n,\a,\b)f_1(\pi n,\a,\b)
+\frac{(-1)^n \p_z^2 f_0(\pi n,\alpha,\beta)f_0^2(\pi n,\a, \b)}{2}
\\
&-f_0^2(\pi n,\a, \b)+f_0(\pi n,\a,\b)
 \p_z f_1(\pi n,\a,\b)+(-1)^nf_2(\pi n,\a,\b).
\end{aligned}
\end{equation}

We can truncate the series in (\ref{prom11z}) and estimate the remainder by means of (\ref{4.19}). This gives the spectral asymptotics for the eigenvalues $\l_n(\a,\b).$ The examples of such asymptotics are provided in our next main result; in order to formulate, we need additional notation:
\begin{equation}\label{Gnj}
\begin{aligned}
G_{n, 1}(\alpha, \beta) &= \cos\pi n\alpha\int\limits_0^{1-\alpha}V(t)\,dt+\cos\pi n\beta\int\limits_\beta^{1}Q(t)\,dt, \\
G_{n, 2}(\alpha, \beta) &= (\alpha+1)\sin\pi n\alpha\int\limits_0^{1-\alpha}V(t)\,dt+(\beta-1)\sin\pi n\beta\int\limits_\beta^{1}Q(t)\,dt, \\
G_{n, 3}(\alpha, \beta) &= (\alpha+1)^2\cos\pi n\alpha\int\limits_0^{1-\alpha}V(t)\,dt+(\beta-1)^2\cos\pi n\beta\int\limits_\beta^{1}Q(t)\,dt,\\
G_{n, 4}(\alpha, \beta) &= -\big(V(1-\alpha)+V(0)\big)\sin\pi n\alpha-\big(Q(1)+Q(\beta)\big)\sin\pi n\beta, \\
G_{n, 5}(\alpha, \beta) &= (\alpha-1)\big(V(1-\alpha)+V(0)\big)\cos\pi n\alpha+(\beta-1)\big(Q(1)+Q(\beta)\big)\cos\pi n\beta, \\
G_{n, 6}(\alpha, \beta) &= V(1-\alpha)\bigg(-\sin\pi n\alpha\int\limits_0^{1-\alpha}V(t)\,dt
+\sin\pi n\beta\int\limits_\beta^{1}Q(t)\,dt\bigg)\sin\pi n\alpha.
\end{aligned}
\end{equation}

\begin{theorem}\label{th2}
For $n\geqslant n_0$ the eigenvalues $\lambda_n(\a,\b)$  possess
the asymptotics
\begin{equation}\label{mainasymptinfty}
\lambda_n(\a,\b)= \pi^2n^2+2\int\limits_0^{1-\alpha}V(t)\cos \pi nt\cos \pi n(t+\alpha)\,dt
+2\int\limits_\beta^1Q(t)\cos \pi nt\cos \pi n(t-\beta)\,dt+O(n^{-1}),
\end{equation}
as $n\to +\infty$. If $V$, $Q\in W^2_\infty (0, 1)$, then
\begin{equation}\label{lambdan1sm}
\lambda_n = \pi^2n^2+G_{n, 1}(\alpha, \beta)+\frac{\Lambda_{n, 1}(\alpha, \beta)}{2\pi n}
+\frac{\Lambda_{n, 2}(\alpha, \beta)}{4\pi^2n^2}+O(n^{-3}),
\end{equation}
as $n\to +\infty$, where
\begin{align*}
\Lambda_{n, 1}(\alpha, \beta) &= G_{n, 4}(\alpha, \beta)-G_{n, 1}(\alpha, \beta)G_{n, 2}(\alpha, \beta), \\
\Lambda_{n, 2}(\alpha, \beta) &= \frac{G_{n, 1}^3(\alpha, \beta)}{6}+G_{n, 1}(\alpha, \beta)G_{n, 2}^2(\alpha, \beta)
-\frac{G_{n, 3}(\alpha, \beta)G_{n, 1}^2(\alpha, \beta)}{2} \\
&\phantom{123}-2G_{n, 1}^2(\alpha, \beta)-G_{n, 1}(\alpha, \beta)G_{n, 5}(\alpha, \beta)
-G_{n, 2}(\alpha, \beta)G_{n, 4}(\alpha, \beta)+4G_{n, 6}(\alpha, \beta) \\
&\phantom{123}+\big(V'(1-\alpha)-V'(0)\big)\cos\pi n\alpha+\big(Q'(1)-Q'(\beta)\big)\cos\pi n\beta \\
&\phantom{123}-\int\limits_0^{1-\alpha}V''(t)\cos\pi n(2t+\alpha)\,dt-\int\limits_\beta^{1}Q''(t)\cos\pi n(2t-\beta)\,dt.
\end{align*}
The error terms in the above asymptotics are  uniform in $\alpha$, $\beta$, and $n$.
\end{theorem}

The third main result describes the Bari basis property of the eigenfunctions and generalized eigenfunctions of the
operator $\mathcal{H}(\alpha, \beta)$ in the space  $L^2(0, 1)$. We recall that the Bari basis is one generated by
systems of projections, which are quadratically close to complete and minimal systems of orthogonal projections.
\begin{theorem}\label{th3}
The system of eigenfunctions and generalized eigenfunctions of the operator
$\mathcal{H}(\alpha, \beta)$ forms the Bari basis in the space $L^2(0, 1)$.
\end{theorem}

Now we briefly discuss  our problem and main results. The main feature of the operator $\Op(\a,\b)$ is its free term described by the operator $\cB(\a,\b)$, which involves two translation operators, see (\ref{B}). This makes the operator $\Op(\a,\b)$ non--local. In order to define properly the translation operators, we follow the scheme suggested in \cite{Skub_Book}, \cite{SkubUMN} and this explains the definition of operator $\cB(\a,\b)$ in terms of the continuation operator $\cE$ and the restriction operator $\cR.$ The operator $\cB(\a,\b)$ also involves two potentials, $V$ and $Q,$ which are complex--valued functions from $L_\infty(0,1).$

Our first preliminary result formulated in Theorem~\ref{th0} describes the basic properties of operator $\mathcal{H}(\alpha, \beta)$. Namely, this is an $m$--sectorial operator with compact resolvent and its spectrum consists of a countably many eigenvalues with the only accumulation point at infinity. The resolvent of this operator and its eigenvalues are continuous in $\a$ and $\b.$

Our main result is formulated in Theorem~\ref{th1}. It states that the sufficiently large eigenvalues can be represented by the formula (\ref{prom11z}) in terms of the convergent series. How large should the eigenvalue be? The answer is \textit{uniform  in $\a$ and $\b$}, namely, this representation holds for all eigenvalues with the index $n\geqslant n_0,$ where $n_0$ is some fixed number \textit{independent of $\a$ and $\b$.} This can be equivalently formulated as follows: there exists a fixed ball in the complex plane independent of $\a$ and $\b$ such that all eigenvalues located outside this ball can be represented by the formula (\ref{prom11z}).

The main ingredient of formula (\ref{prom11z}) is of course the series. The  terms of this series are negative powers of $n$ with the coefficients depending on $\a,$ $\b,$ and $n.$ The main property of these coefficients is 
the estimates (\ref{4.3}) and in fact, these estimates ensure the absolute and uniform in $\a,$ $\b$ and $n$ convergence of this series. This series resembles the Taylor series, in which the quotient $\frac{1}{\pi n}$ plays the role of the small increment of the independent variable, but the main difference is the dependence of the coefficients on $n.$  Once the coefficients $\rho_i$ are known, the series converges, and the formula (\ref{prom11z}) provides the \textit{explicit expression} for the eigenvalues. The coefficients $\rho_i$ can be indeed found explicitly, namely, there are explicit recurrent formulas (\ref{polDi}) for them, which express the coefficients via the potentials $V$ and $Q.$ The formulas are  quite cumbersome and it is not a trivial problem to find some direct, non--recurrent formulas in terms of $i,$ $\a,$ $\b,$ $V,$ and $Q.$ But nevertheless, we stress again, once these coefficients are found, we have \textit{explicit and exact formulas (\ref{prom11z}) for the eigenvalues} $\l_n(\a,\b).$

To the best of our knowledge, the representations similar to (\ref{prom11z}) are not known even for the classical Sturm~---~Liouville problems. Namely, while the spectral asymptotics for classical problems were studied very well, see, for instance, \cite{Marchenko}, \cite{LS}, the convergence of the asymptotic series is not known. Moreover, the classical  asymptotic series usually diverges and the reason is as follows. Provided the functions $V$ and $Q$ are smooth enough, one can use the approach from the stationary phase method and integrate by parts in the definitions of the coefficients $\rho_i$ and to obtain the asymptotics for these coefficients in the negative powers of $n.$ Substituting these asymptotics into the series in (\ref{prom11z}), we get the spectral asymptotics only in negative powers of $n$ with fixed coefficients. However, in such asymptotics various remainders come from the aforementioned integration by parts, and exactly such remainders spoil the convergence and give the divergent asymptotics series. In our approach we keep the coefficients $\rho_i$ as they are and do not replace them by their asymptotic expansions. This allows us to preserve the convergence of series in (\ref{prom11z}) and this is the main difference from the classical results we know.

The explicit formulas (\ref{prom11z}) for the eigenvalues can be useful for many applications and further studies, and one of very natural application is as follows. While calculating all coefficients $\rho_i$ is a quite non--trivial problem, by the recurrent formulas (\ref{polDi}) we can find any prescribed number of the coefficients. Then we additionally have the explicit estimate (\ref{4.19}) for the remainder of series and hence, the partial sum of this series provides an effective approximation for the eigenvalues $\l_n(\a,\b).$ In this way we can find the spectral asymptotics for the eigenvalues with the uniform in $\a$ and $\b$ estimates for the error term, or, shortly, \textit{uniform spectral asymptotics.} We demonstrate this possibility for $N=2$ and $N=4$ in Theorem~\ref{th2}, in which we provide the uniform spectral asymptotics for the eigenvalues $\l_n(\a,\b).$  It should be said that these asymptotics are similar in spirit to the spectral asymptotics for the  Dirichlet condition recently obtained in   \cite{Borisov_PolyakovDAN}, \cite{Borisov_PolyakovIZV}, \cite{Borisov_PolyakovUMJ}, but here they are more complicated due to the presence of two translations.

The asymptotics (\ref{mainasymptinfty}) is general  in the sense that it is valid for essentially bounded potentials $V$ and $Q.$ Trying to rewrite the formulas for $\rho_1$ and $\rho_2$ in (\ref{coeffD}), we get quite bulky formulas, and to simplify them, we suppose in addition that the potentials $V$ and $Q$  are smooth enough. This allows us to make the above discussed integration by parts in the coefficients $\rho_1$ and $\rho_2$ and to get the asymptotics (\ref{lambdan1sm}).  Let us discuss the latter asymptotics in more detail.

As it has been said above, the representations (\ref{prom11z}) hold for all eigenvalues of the operator $\Op(\a,\b)$ outside some fixed ball independent of $\a$ and $\b$ and this is why the same is true for the asymptotics (\ref{mainasymptinfty}), (\ref{lambdan1sm}). This means that both these asymptotics describe the behavior of \textit{entire ensemble of eigenvalues} except for finitely many of them uniformly in the parameters $\a$ and $\b.$

For $\alpha=\beta=0$ the asymptotics (\ref{mainasymptinfty}) becomes the classical one for the Schr\"odinger operator on the segment
(see, for example, \cite[Ch.~1, Sec.~5]{Marchenko}):
$$
\lambda_n(0,0)=\pi^2n^2+ \int\limits_0^1\big(V(t)+Q(t)\big)\,dt+O(n^{-2}).
$$
If $\alpha=\beta=1$, then $G_{n, 1}(\alpha, \beta)=\Lambda_{n, 1}(\alpha, \beta)=\Lambda_{n, 2}(\alpha, \beta)=0$ and we again get classical case $\lambda_n(1,1)=\pi^2n^2$.

For  $\alpha,\, \beta\in (0, 1)$ the leading terms of  the asymptotics (\ref{mainasymptinfty}),  (\ref{lambdan1sm}) depend on  $n$ via the sines and cosines of $\pi n \a$ and $\pi n \b$ and there is also an  additional  dependence on $\a$ and $\b,$ see (\ref{Gnj}). As $\a\to0$ and $\b\to0,$
then Theorems~5 and~6 from \cite{Borisov_Polyakov} state that the eigenvalues of $\l_n(\a,\b)$ converge to $\l_n(0,0).$ However, this statement holds only for each finite set of the eigenvalues and this convergence is not uniform in $n.$
  This is reflected already by the asymptotics (\ref{mainasymptinfty}), (\ref{lambdan1sm}) since the mentioned sines and cosines of $\pi n\a$ and $\pi n\b$ do not demonstrate the uniform in $n$ convergence as $\a\to0$ and $\b\to0.$ In particular, it follows from
the asymptotics (\ref{mainasymptinfty}) that
for sufficiently large $n$ the eigenvalues $\l_n(\a,\b)$
 are localized
along the curve in the complex plane
$$
\Gamma_{\alpha, \beta}:=\big\{z=t^2+v(\a)\cos t\alpha+q(\beta)\cos t\beta\big\}\subset\mathds{C}, \qquad
v(\alpha):=\int\limits_0^{1-\alpha}V(t)\,dt, \qquad q(\beta):=\int\limits_\beta^{1}Q(t)\,dt,
$$
in the vicinity of the points  at this curve associated with   $t=\pi n$. The curve $\G_{\a,\b}$ resembles a graph of a doubly periodic function; as $\a$ and $\b$ grow from $0$ to $1,$ the period of oscillations  descrease and we see more oscillations along the curve. Such high--frequency phenomenon is generated solely by the translations  in the operator $\Op(\a,\b)$. This phenomenon seems to be quite general and independent of the type of boundary conditions since the similar phenomenon was found for the Dirichlet conditions as well in \cite{Borisov_PolyakovDAN}, \cite{Borisov_PolyakovIZV}, \cite{Borisov_PolyakovUMJ}.

Our third result, Theorem~\ref{th2}, states the basicity of the eigenfunctions and generalized eigenfunctions of the operator $\Op(\a,\b),$ namely, these functions form  a Bari basis in $L^2(0, 1)$.  This is an important result since the known present results concern only the completeness and minimality of systems of eigenfunctions and
generalized eigenfunctions for differential--difference operators, and there are no even results on the  classical Riesz basicity.

\section{Basic properties of   operator $\mathcal{H}(\alpha, \beta)$}\label{sec3*}

In this section we prove Theorem~\ref{th0}. Using the results of \cite[Sect.~6, Thm.~4]{Borisov_Polyakov}, we obtain
the sesquilinear form $\mathfrak{h}(\alpha, \beta)$ is closed, sectorial, and its numerical range is contained
in the sector $\{z\in\mathds{C}:\, |\IM z|\leqslant c_{3}(\RE z-c_{4})\}$,
where $c_{3}$ and $c_{4}$ are some positive constants independent of $\alpha$ and $\beta$. By the first
representation Theorem \cite[Ch. V\!I, Sect.~2.1, Thm.~2.1]{Kato}, there exists
associated $m$--sectorial operator $\widetilde{\mathcal{H}}(\alpha, \beta)$. We are going to show that $\widetilde{\mathcal{H}}(\alpha, \beta)=\mathcal{H}(\alpha, \beta)$.
Let $u\in\mathfrak{D}(\mathcal{H}(\alpha, \beta))$, $v\in\mathfrak{D}(\mathfrak{h}(\alpha, \beta))$.
Integrating by parts, we obtain
\[
(\mathcal{H}(\alpha, \beta)u, v)_{L^2(0, 1)}=\mathfrak{h}(\alpha, \beta)[u, v].
\]
Therefore, $\mathfrak{D}(\mathcal{H}(\alpha, \beta))\subseteq\mathfrak{D}(\widetilde{\mathcal{H}}(\alpha, \beta))$
and the operator $\widetilde{\mathcal{H}}(\alpha, \beta)$ is a continuation of operator
$\mathcal{H}(\alpha, \beta)$. Thus, to prove the coincidence of the operators
$\mathcal{H}(\alpha, \beta)$ and $\widetilde{\mathcal{H}}(\alpha, \beta)$ it is
to verify that $\mathfrak{D}(\mathcal{H}(\alpha, \beta))
\supseteq\mathfrak{D}(\widetilde{\mathcal{H}}(\alpha, \beta))$.
By the first representation Theorem~\cite[Ch.~V\!I, Sect.~2.1, Thm.~2.1]{Kato},
the domain $\widetilde{\mathcal{H}}(\alpha, \beta)$ consist of functions
$u\in\mathfrak{D}(\mathfrak{h}(\alpha, \beta))$ such that there exists a function
$h\in L^2(0, 1)$ satisfying the condition
\begin{equation*}
\mathfrak{h}(\alpha, \beta)[u, v]=(h, v)_{L^2(0, 1)}\quad\text{for each}\quad
v\in\mathfrak{D}(\mathfrak{h}(\alpha, \beta)).
\end{equation*}
Let $u$ be one of these functions. It follows from definition of the
sesquilinear form \eqref{2.1} that the identity can rewrite as
\begin{equation}\label{3.2}
(u', v')_{L^2(0, 1)}=(\widetilde{h}, v)_{L^2(0, 1)}\quad\text{for each}
\quad v\in\mathfrak{D}(\mathfrak{h}(\alpha, \beta)),
\end{equation}
where
\[
\widetilde{h}:=h-V\mathcal{R}\mathcal{T}(\alpha)\cE u
-Q\mathcal{R}\mathcal{T}(-\beta)\cE u\in L^2(0, 1).
\]
The formula \eqref{3.2} implies that the function $u$ is generalized solution to the problem
\[
-u''=\widetilde{h} \quad \text{in} \quad (0, 1), \qquad u'(0)=u'(1)=0.
\]
Using the classical theorems on increasing the smoothness of solutions to elliptic
boundary value problems, we deduce $\mathfrak{D}(\mathcal{H}(\alpha, \beta))=
\{u\in W_2^2(0, 1): u'(0)=u'(1)=0\}$. Therefore,
$\mathfrak{D}(\widetilde{\mathcal{H}}(\alpha, \beta))\subseteq
\mathfrak{D}(\mathcal{H}(\alpha, \beta))$. This gives
$\mathfrak{D}(\widetilde{\mathcal{H}}(\alpha, \beta))=\mathfrak{D}(\mathcal{H}(\alpha, \beta))$ and
$\widetilde{\mathcal{H}}(\alpha, \beta)=\mathcal{H}(\alpha, \beta)$.
It follows from the properties of the
sesquilinear form $\mathfrak{h}(\alpha, \beta)$ that the half--plane $\RE \lambda\leqslant c_{4}-1$
belongs to the resolvent set of the operators $\mathcal{H}(\alpha, \beta)$ for each
$\alpha$, $\beta\in [0, 1]$ and the estimate
\begin{equation}\label{3.3}
\RE \hf(\a,\b)[u,u]+\RE\l \|u\|_{L_2(0,1)}^2\geqslant \|u\|_{W_2^1(0,1)}^2
\end{equation}
holds. The resolvent of the operator
$\mathcal{H}(\alpha, \beta)$ is a bounded operator acting in $L^2(0, 1)$ and by the
inverse mapping theorem this operator is bounded from $L^2(0, 1)$ into $W_2^1(0, 1)$. It follows from
the compactness of the embedding $W_2^1(0, 1)\subset L^2(0, 1)$ that the
resolvent of the operator $\mathcal{H}(\alpha, \beta)$ is compact operator.
Using \cite[Ch.~3, \S~6.8, Theorem~6.29]{Kato}, we obtain that
the spectrum of the operator $\mathcal{H}(\alpha, \beta)$ consists of countably
many discrete eigenvalues with the only accumulation point at infinity.

The estimate (\ref{3.3}) implies that  for $\RE\l\leqslant c_4-1$
\begin{equation*}
\|(\Op(\a,\b)-\l)^{-1}\|_{L_2(0,1)\to W_2^1(0,1)}\leqslant 1.
\end{equation*}
In view of the equation for the resolvent we then obtain the first estimate in (\ref{2.3}).
To prove the second estimate in (\ref{2.3}), we observe the easily verified identity
\begin{equation}\label{3.4}
\begin{aligned}
(\Op(\a_1,\b_1)-\l)^{-1}-(\Op(\a_2,\b_2)-\l)^{-1}=&(\Op(\a_2,\b_2)-\l)^{-1}
\big(V\cR(\cT(\a_1)-\cT(\a_2))\cE
\\
&+Q\cR(\cT(-\b_1)-\cT(-\b_2))\cE\big)
(\Op(\a_1,\b_1)-\l)^{-1}.
\end{aligned}
\end{equation}
It is straightforward to verify  that for all $\a,\,\b\in[0,1],$ $\a\leqslant \b,$ $u\in W_2^1(0,1)$ we have
\begin{equation*}
\big(\cR(\cT(\a)-\cT(\b))\cE u\big)(x)=\left\{
\begin{aligned}
& 0, && x\leqslant \a,
\\
u(x&-\a), && \a<x\leqslant \b,
\\
\int\limits_{x-\b}^{x-\a}& u'(t)\,dt,\quad && x\geqslant \b.
\end{aligned}
\right.
\end{equation*}
Using now the embedding $W_2^1(0,1)\subset C[0,1],$ we find
\begin{align*}
\|\big(\cR(\cT(\a)-\cT(\b))\cE u\big)\|_{L_2(0,1)}^2 \leqslant &C |\b-\a|\|u\|_{W_2^1(0,1)}^2 + C |\b-\a|\int\limits_{\b}^{1}\,dx\int\limits_{x-\b}^{x-a}  |u'(t)|^2 dt
\\
\leqslant & C|\b-\a| \|u\|_{W_2^1(0,1)}^2,
\end{align*}
where $C$ are some constants independent of $\a,$ $\b$ and $u.$ The obtained estimate,  identity (\ref{3.4}) and the first inequality in (\ref{2.3}) imply the second inequality in (\ref{2.3}).   This inequality means that the resolvent of operator $\Op(\a,\b)$ is continuous in $(\a,\b)\in[0,1]^2.$ Hence, the same is true for the eigenvalues of this operator.  The proof is complete.

\section{Representations for eigenvalues}\label{sec2}

In this section we prove Theorem~\ref{th1}. The eigenvalues of the operator $\mathcal{H}(\alpha, \beta)$ solve
the boundary value problem
\begin{align}
-y''(x) &+ V(x)y(x+\alpha)+Q(x)y(x-\beta) =\lambda y(x), \quad x\in(0,1),\label{bp1.1}
\\
y'(0) &= y'(1)=0, \label{bp2.1}
\end{align}
where $\lambda$ is a spectral parameter and the function $y$ is supposed to be continued by zero outside the segment $[0, 1]$
and after that the translation is restricted to the segment $[0, 1]$. Throughout this section we represent the eigenvalues
of the above problem as $\lambda=z^2$, $z\in\mathds{C}$, $\RE z\geqslant 0$.

\begin{lemma}\label{lm4.1}
Let $\l$  be an eigenvalue of the operator $\mathcal{H}(\alpha, \beta)$. Then the corresponding value $z\in\mathds{C}$
satisfies $z\in\Pi,$ where
\begin{equation}\label{4.9}
\Pi:=\{z\in\mathds{C}:\,\RE z^2\geqslant -c_{5},\ |\IM z^2|\leqslant c_{5}\},\qquad \Pi\subset
 \Big\{z\in\mathds{C}:\,\RE z\geqslant 0,\ |\IM z|\leqslant \sqrt{2}c_{5}\Big\}.
\end{equation}
\end{lemma}

\begin{proof}
A normalized in $L^2(0, 1)$ eigenfunction $\psi$ associated with $\lambda$ satisfies the integral identity
\[
\|\psi'\|_{L^2(0, 1)}^2+(\mathcal{B}(\alpha, \beta)\psi, \psi)_{L^2(0, 1)}=\lambda\|\psi\|_{L^2(0, 1)}^2.
\]
Taking the imaginary part of this identity, in view of the normalization of the function $\psi$ and the boundedness
of the operator $\mathcal{B}(\alpha, \beta)$ we obtain
\[
|\IM \lambda|=\Big|\IM (\mathcal{B}(\alpha, \beta)\psi, \psi)_{L^2(0, 1)}\Big|\leqslant
\|\mathcal{B}(\alpha, \beta)\|_{L^2(0, 1)}\|\psi\|_{L^2(0, 1)}^2= \|\mathcal{B}(\alpha, \beta)\|_{L^2(0, 1)}.
\]
It follows from \cite[Lemma~1]{Borisov_Polyakov} and \eqref{B} that
\begin{equation*}
\|\mathcal{B}(\alpha, \beta)\|_{L^2(0, 1)}\leqslant c_{5},
\end{equation*}
where $c_{5}$ is some positive constant independent of $\alpha$ and $\beta$. Therefore,
\begin{equation*}
|\IM \lambda|\leqslant c_{5}, \qquad  \RE \lambda\geqslant -c_{5},
\end{equation*}
where $c_{5}$ is some constant independent of $\alpha$, $\beta$, and $\lambda$. This is why for all eigenvalues of the
operator $\mathcal{H}(\alpha, \beta)$ the corresponding values $z$ are located in $\Pi$.
A number $z$ belongs to $\Pi$ if
\begin{equation*}
(\RE z)^2\geqslant (\IM z)^2-c_{5},\qquad |\IM z||\RE z|\leqslant \frac{c_{5}}{2}.
\end{equation*}
These two inequalities imply
\begin{equation*}
(\IM z)^2-c_{5}\leqslant \frac{c_{5}^2}{4|\IM z|^2}
\end{equation*}
and this inequality holds only if
\begin{equation*}
|\IM z|^2 \leqslant \frac{\sqrt{2}+1}{2}c_{5}\leqslant 2c_{5}.
\end{equation*}
Since $\RE z$ is non--negative, we arrive at the embedding for $\Pi$ in \eqref{4.9}. The proof is complete.
 \end{proof}

In order to study the problem \eqref{bp1.1}, \eqref{bp2.1}, we consider an auxiliary Cauchy problem for
the equation in \eqref{bp1.1}, namely,
\begin{equation}\label{4.1}
\begin{aligned}
-&\varphi''(x, z)+ V(x)\varphi(x+\alpha, z)+Q(x)\varphi(x-\beta, z)=z^2 \varphi(x, z), \quad x\in(0, 1),
\\
&\varphi(0, z)=1, \qquad \varphi'(0, z)=0.
\end{aligned}
\end{equation}
The solvability of this  Cauchy problem is described in the next lemma; we recall that the operators
$\mathcal{M}(z,\alpha, \beta)$ and $\cL(z,\alpha, \beta)$ are defined in  \eqref{OperM}.

\begin{lemma}\label{lm4.2}
The Cauchy problem \eqref{4.1} is uniquely solvable for each $z\in\Pi$ obeying the bound
\begin{equation}\label{4.6}
\RE z\geqslant c_{5} e^{\sqrt{2}c_{5}}+1.
\end{equation}
The solution is represented by the series
\begin{equation}\label{4.7}
\varphi(\,\cdot\,, z)= \sum\limits_{j=0}^{\infty} \frac{(\mathcal{M}(z,\alpha, \beta))^j\Phi(\,\cdot\,,z)}{z^j},\qquad
\Phi(x,z):=\cos zx,
\end{equation}
and this series converges in the sense of $C[0,1]$--norm uniformly in $\alpha$, $\beta\in[0, 1]$ and $z\in\Pi$
obeying \eqref{4.6}. The norms of the operators $\mathcal{M}(z,\alpha, \beta)$ and
$\cL(z,\alpha, \beta)$ obey the estimate
\begin{equation}\label{4.13}
\|\mathcal{M}(z,\alpha, \beta)\varphi\|_{C[0, 1]}
\leqslant  c_{5} e^{\sqrt{2}c_{5}}\|\varphi\|_{C[0, 1]},\qquad
\|\cL(z,\alpha, \beta)\varphi\|_{C[0, 1]}
\leqslant  c_{5} e^{\sqrt{2}c_{5}}\|\varphi\|_{C[0, 1]}.
\end{equation}
The identity
\begin{equation}\label{4.10}
\varphi'(\,\cdot\,, z) = \Phi'(\,\cdot\,,z) +  \cL(z,\alpha, \beta)\sum\limits_{j=0}^{\infty}
\frac{(\mathcal{M}(z,\alpha, \beta))^j\Phi(\,\cdot\,,z)}{z^j}
\end{equation}
holds, where the series converges in the sense of $C[0,1]$--norm uniformly in $\alpha$, $\beta\in[0, 1]$ and
$z\in\Pi$ obeying \eqref{4.6}. The sum of this series and the operators $\mathcal{M}(z,\alpha, \beta)$,
$\cL(z,\alpha, \beta)$ are holomorphic in $z\in\Pi$ obeying~\eqref{4.6}.

The Cauchy problem for the equation in \eqref{4.1} with the homogeneous initial conditions has only the zero solution.
\end{lemma}

\begin{proof}
In a standard way we verify that this Cauchy problem is equivalent to the integral equation
\begin{equation}\label{bp2.4}
\varphi(x, z) = \cos zx+\frac{1}{z}\int\limits_0^{x}(\mathcal{B}(\alpha, \beta)\varphi)(t, z)\sin z(x-t)\,dt,
\end{equation}
which is considered in the space $C[0, 1]$. We rewrite this equation as
\begin{equation}\label{4.4}
\Big(\cI-\frac{1}{z}\mathcal{M}(z,\alpha, \beta)\Big)\varphi(\,\cdot\,, z)=\Phi(\,\cdot\,,z).
\end{equation}

We take $z\in\Pi$ and for an arbitrary $\varphi\in C[0,1]$ we get
\begin{align*}
\|\mathcal{M}(z,\alpha, \beta)\varphi\|_{C[0, 1]} &=
\max\limits_{x\in [0, 1]}\bigg|\int\limits_0^x(\mathcal{B}(\alpha, \beta)\varphi)(t,z)\sin z(x-t)\,dt\bigg| \\
&\leqslant \max\limits_{x\in [0, 1]}\bigg|\int\limits_0^1 |(\mathcal{B}(\alpha, \beta)\varphi)(t, z)||\sin z(x-t)|\,dt\bigg|\\
&\leqslant c_{5} e^{\sqrt{2}c_{5}}\|\varphi\|_{C[0,1]}
\end{align*}
and this proves the first in estimate \eqref{4.13}. The second estimate can be proved in the same way. Hence,
the operator in the left hand side in \eqref{4.4} is invertible provided $\RE z$ is large enough, namely, for
$z$ obeying the condition \eqref{4.6} and for such values $z$ we can solve Equation \eqref{4.4}:
\begin{equation}\label{4.5}
\varphi(\,\cdot\,, z)=\Big(\mathcal{I}-\frac{1}{z}\mathcal{M}(z,\alpha, \beta)\Big)^{-1}\Phi(\,\cdot\,,z).
\end{equation}
Hence, for such $z$ Equation \eqref{bp2.4} and the Cauchy problem \eqref{4.1} are uniquely solvable and the
solution is given by the formula \eqref{4.5}. Moreover, expanding the right hand side in \eqref{4.5} into the Neumann series,
we obtain the representation \eqref{4.7}. The uniform convergence of this series is ensured by the estimate~\eqref{4.13}.

It follows from Equation~\eqref{bp2.4} that
\begin{equation*}
\varphi'(x, z) = -z\sin zx+(\cL(z,\alpha, \beta) \varphi)(x, z).
\end{equation*}
Using the representation \eqref{4.7}, we immediately arrive at the formula \eqref{4.10}, and the uniform convergence
of the series in its right hand side is owing to the estimates \eqref{4.13}. The operators $\mathcal{M}(z,\alpha, \beta)$,
$\cL(z,\alpha, \beta)$ are obviously holomorphic in $z\in\Pi$ obeying \eqref{4.6} and this implies the
holomorphy in the same $z$ for the series in \eqref{4.10}.

It remains to prove that the Cauchy problem for the equation in \eqref{4.1} with the homogeneous initial conditions has only
the zero solution. Indeed, such problem is equivalent to the integral equation
\begin{equation*}
\varphi(x,z) = \frac{1}{z}\int\limits_0^{x}(\mathcal{B}(\alpha, \beta)\varphi)(t, z)\sin z(x-t)\,dt,
\end{equation*}
which can be rewritten as
\begin{equation*}
\Big(\cI-\frac{1}{z}\mathcal{M}(z,\alpha, \beta)\Big)\varphi(\,\cdot\,, z)=0.
\end{equation*}
The latter equation possesses only the zero solution.
\end{proof}

We return back to the problem \eqref{bp1.1}. A non--trivial solution to this problem necessary has a non--zero
derivative at $x=0$ since otherwise it would solve the Cauchy problem the equation in \eqref{4.1} with the homogeneous
initial conditions and then it would be identically zero. Hence, up to a multiplication by an appropriate non--zero
constant, this non--trivial solution solves the Cauchy problem \eqref{4.1}. This is why $z$ produces an eigenvalue of
the operator $\mathcal{H}(\alpha, \beta)$ if and only if the solution to the Cauchy problem satisfies the Neumann condition
at $x=1$, that is,
\begin{equation}\label{4.12}
\varphi'(1, z)=0.
\end{equation}
This is the equation for the values of $z\in\Pi$ obeying \eqref{4.6} and producing the eigenvalues of the
operator $\mathcal{H}(\alpha, \beta)$ by the formula $\lambda=z^2$. Let us study the solvability of this equation.

The identitites \eqref{4.10} and \eqref{4.12} give
\begin{equation}\label{4.17'}
-z\sin z +(\cL(z,\alpha, \beta)\varphi)(1, z)=0.
\end{equation}
Since the function in the left part of this equality is entire in $z$ and is not a polynomial, it has infinitely
many zeros and this means that \eqref{4.17'} possesses infinitely many roots. In what follows, we study only the roots
of Equation \eqref{4.17'} with sufficiently large absolute values, that is,
$$
|z|\geqslant R,
$$
where $R$ is sufficiently large number independent of $\alpha$ and $\beta$. Considering roots in the strip
$\Pi$ obeying~\eqref{4.6}, we partition this strip into rectangles
$$
\Pi_n:=\Big\{z: \RE z\in \Big(\pi n-\frac{\pi}{2}, \pi n+\frac{\pi}{2}\Big), \;
\IM z\in (-\sqrt{2}c_{5}, \sqrt{2}c_{5})\Big\}, \qquad n\geqslant N:=\left[\frac{R}{\pi}\right],
$$
where $[\,\cdot\,]$ is the integer part of a number. In each such rectangle the function $z\mapsto z\sin z$ possesses
a single zero $z=\pi n$, which is simple. In view of the obvious identity
$$
|\sin z|^2=\sin^2\RE z+\sinh^2\IM z,
$$
on the boundary of the rectangles, we have the following lower bound:
$$
|z\sin z|\geqslant Rc_{6} \quad \text{on} \quad \partial \Pi_n, \quad c_{6}:=\min\{1, \sinh \sqrt{2}c_{5}\}.
$$
Choosing $R=2c_{6}^{-1}c_{5} e^{\sqrt{2}c_{5}}\|\varphi\|_{C[0, 1]}$, by the above estimate and the second
estimate in \eqref{4.13} we obtain
$$
|z\sin z|>|(\cL(z,\alpha, \beta)\varphi)(1, z)| \quad \text{on} \quad \partial\Pi_n, \quad n\geqslant N.
$$
Using the Rouche theorem, we conclude that each rectangle $\Pi_n$ possesses exactly one simple root of Equation \eqref{4.17'}.

Our next step is to study the roots of Equation \eqref{4.17'}. We rewrite this equation as
\begin{equation}\label{eqforz}
\sin z-\frac{1}{z}\sum\limits_{j=0}^{\infty}\frac{f_j(z,\alpha, \beta)}{z^{j}}=0,
\end{equation}
where $f_j(z, \alpha, \beta)$ is defined by \eqref{fj}. It follows from the properties of the operators
$\mathcal{M}(z,\alpha, \beta)$ and $\cL(z,\alpha, \beta)$ stated in Lemma~\ref{lm4.2} that the
functions $f_j(z,\alpha, \beta)$ are holomorphic in $z$ obeying \eqref{4.6} and due to \eqref{4.13} we see that
\begin{equation}\label{4.14}
 |f_j(z,\alpha, \beta)|\leqslant c_{5}c_{7}^j,\qquad j\geqslant 0,
\end{equation}
where $c_{5}$ and $c_{7}$ are some positive constants independent of $z$, $\alpha$, and $\beta$.

Let $\xi:=z-\pi n$ and $\eta:=\sin\xi$. Then we rewrite Equation \eqref{eqforz} as
$$
\eta-\sum\limits_{j=0}^{\infty}
\frac{f_j (\arcsin\eta+\pi n,\alpha, \beta)}{(\arcsin\eta+\pi n)^{j+1}}=0.
$$
Finally, we put $\mu:=(\pi n)^{-1}$. Then we rewrite the latter equation as
\begin{equation}\label{equateta}
\eta-\mu F(\eta, \mu, n,\a,\b)=0,
\end{equation}
where
\begin{equation}\label{eqmaineta}
F(\eta, \mu, n,\a,\b)=\sum\limits_{j=0}^{\infty}
\frac{\mu^j f_j (\pi n+\arcsin\eta,\a,\b)}{(1+\mu\arcsin\eta)^{j+1}}.
\end{equation}
Owing to the estimates (\ref{4.14}), the series in the above formula converges uniformly in \begin{equation*}
(\eta,\mu)\in\Om:=\big\{(\eta,\mu):\, |\eta|\leqslant c_{11},\ |\mu|\leqslant c_{11}\big\},
\end{equation*}
where $c_{11}$ is independent of $\eta,$ $\mu$, and $n$. Hence, the function $F$ is holomorphic in
$(\eta,\mu)\in\Omega$ and it is also bounded uniformly in $(\eta,\mu)\in\Omega$ and $n\in\mathds{N}$.
Therefore, the function $F$ can be represented by its Taylor series
\begin{equation}\label{4.16}
F (\eta, \mu, n,\a,b)=\sum\limits_{\gamma\in\mathds{Z}_+^2}A_\gamma(n,\a,\b)\eta^{\gamma_1}
\mu^{\gamma_2}, \qquad \gamma=(\gamma_1, \gamma_2),
\end{equation}
and the coefficients $A_\gamma$ obey the estimates
\begin{equation}\label{estAgamma}
|A_\gamma(n,\a,\b)|\leqslant c_{8}c_{9}^{|\gamma|},
\end{equation}
where $c_{8}$ and $c_{9}$ are some positive constants independent of $\alpha$, $\beta$, $n$, and $|\gamma|=\gamma_1+\gamma_2$.

We prove the solvability of Equation~\eqref{equateta} by applying the implicit function theorem
\cite[Th.~1.3.5, Rem.~1.3.6]{Narasimhan} to Equation~(\ref{equateta}) with the function $F$ from
\eqref{eqmaineta}. As noted above, the function $F$ is holomorphic in $\eta$ and $\mu$ and
it follows from \eqref{eqmaineta} that the left hand side of Equation~\eqref{equateta} vanishes at the point
$(\eta,\mu)=(0, 0)$, while its derivative in $\eta$ at the same point is equal to $1.$ Hence, by the implicit function theorem
for each $n$ there exists a unique solution $\eta=\varphi(\mu, n)$, which is holomorphic in $\mu$ in some neighbourhood
of zero and the size of the neighbourhood, generally speaking, depends on $n$. It also follows from  Equation~\eqref{equateta}
that the Taylor series of $\varphi$ is
\begin{equation}\label{varphiseries}
\varphi(\mu, n,\a,\b)=\mu\sum\limits_{i=0}^\infty B_i(n,\a,\b)\mu^i,
\end{equation}
where $B_i(n,\a,\b)$ are some constants.
This series converges for $|\mu|<r(n,\a,\b),$ where $r(n,\a,\b)$ is some positive number.

\begin{lemma}\label{lm4.3}
There exists a fixed positive constant $c_{10}$   independent of $\alpha$, $\beta$, and $n$ such that
$$
r(n,\a,\b)\geqslant c_{10}>0\quad \text{for all}\quad n.
$$
The coefficients $B_i$, $i\geqslant 0$, satisfy the estimate
\begin{equation}\label{4.24**}
|B_i(n,\a,\b)|\leqslant c_{11} c_{12}^i,\qquad i\geqslant 0,
\end{equation}
where $c_{11}$, $c_{12}$ are positive constants independent of $\alpha$, $\beta$, and $n$.
\end{lemma}
\begin{proof}
According to the Cauchy--Hadamard formula, it is sufficient to
estimate the coefficients $B_i$ uniformly in $n$. In order to do it, we use appropriate majorants for the functions
$F$ and $\varphi$. As the majorant for the former, we employ the holomorphic function
$\widetilde{F}=\widetilde{F}(\eta,\mu)$ defined
by the right hand side of~(\ref{4.16}) with the coefficients $A_\g$ replaced by $c_{8} c_{9}^{|\g|},$ namely,
\begin{equation*}
\widetilde{F}(\eta,\mu):=\sum\limits_{\gamma\in\mathds{Z}_+^2}
c_{8} c_{9}^{|\g|} \eta^{\gamma_1}\mu^{\gamma_2}=\frac{c_{8}}{(1-c_{9}\eta)(1-c_{9}\mu)}.
\end{equation*}
Then we consider   Equation \eqref{equateta} with such a function, and we rewrite this equation as
$$
c_{9}\eta^2(\mu)-\eta(\mu)+\frac{c_{8}\mu}{1-c_{9}\mu}=0.
$$
This equation has the unique solution, which vanishes at $\mu=0,$  and this solution can be found explicitly
$$
\eta(\mu)=\frac{1-\sqrt{1-\frac{4c_{8}c_{9}\mu}{1-c_{9}\mu} }}{2c_{9}}.
$$
It is straightforward to verify that the coefficients of    Taylor series about $\mu=0$ of the functions
\begin{equation*}
\mu\mapsto \frac{4c_{8}c_{9}\mu}{1-c_{9}\mu},\qquad u\mapsto \frac{1-\sqrt{1-u}}{2c_{9}}
\end{equation*}
are real and positive. Hence, the same is true for the function $\eta(\mu)$ as well as for
the function
$$
\widetilde{\varphi}(\mu):=\frac{1-\sqrt{1-\frac{4c_{8}c_{9}\mu}{1-c_{9}\mu}}}{2c_{9}\mu}.
$$
It is clear that the latter function is holomorphic in sufficiently small $\mu$ and by $\tau_i,$ $i\geqslant 0$, we denote its
Taylor coefficients,
\begin{equation*}
\widetilde{\varphi}(\mu)=\sum\limits_{i=0}^{\infty} \tau_i\mu^i,\qquad \tau_0=c_{8},\qquad \tau_i>0.
\end{equation*}
Since the function $\widetilde{\varphi}$ is holomorphic, we have standard bounds for $\tau_i$
\begin{equation*}
0<\tau_i\leqslant c_{11} c_{12}^i,
\end{equation*}
where $c_{11}$ and $c_{12}$ are some positive constants independent of $i,$ and of course, of $\alpha$, $\beta$, and $n$.

Now we are going to show that
\begin{equation}\label{4.24}
|B_i(n,\a,\b)|\leqslant \tau_i,\qquad i\geqslant0.
\end{equation}
We substitute the representation \eqref{varphiseries} into \eqref{eqmaineta} and obtain
\begin{align*}
\sum\limits_{i=0}^\infty B_i(n,\a,\b)\mu^i=&\sum\limits_{\gamma\in\mathds{Z}_+^2}
A_\gamma(n,\a,\b)\mu^{\gamma_2}
\bigg(\sum\limits_{i=0}^\infty B_i(n,\a,\b)\mu^i\bigg)^{\gamma_1}
\\
=&\sum\limits_{i=0}^\infty\mu^i\sum\limits_{\gamma\in\mathds{Z}_+^2}
\sum\limits_{\substack{\varkappa\in\mathds{Z}_+^{\gamma_1} \\ |\varkappa|+\gamma_2=i}}
A_\gamma(n,\a,\b)B_{\varkappa_1}(n,\a,\b)\cdot\ldots\cdot B_{\varkappa_{\gamma_1}}(n,\a,\b),
\end{align*}
where $\varkappa=(\varkappa_1, \dots, \varkappa_{\gamma_1})$. Hence, the coefficients
$B_i$ satisfy the recurrent relations
\begin{equation}\label{maincoeff}
B_i(n,\a,\b)=\sum\limits_{\gamma\in\mathds{Z}_+^2}
\sum\limits_{\substack{\varkappa\in\mathds{Z}_+^{\gamma_1} \\ |\varkappa|+\gamma_2=i}}
A_\gamma(n,\a,\b)B_{\varkappa_1}(n,\a,\b)\cdot\ldots\cdot B_{\varkappa_{\gamma_1}}(n,\a,\b),
\end{equation}
where we adopt that $B_{\varkappa_1}(n,\a,\b)\cdot\dots\cdot B_{\varkappa_{\gamma_1}}(n,\a,\b):=1$
if $\gamma_1=\gamma_2=0$ or $\varkappa=0$.
Then the formula \eqref{maincoeff} expresses the coefficients $B_i$ in terms of $B_0, \dots, B_{i-1}$.

It is clear that $B_0(n,\a,\b)=A_0(n,\a,\b)$ and hence, by \eqref{estAgamma}, we immediately get
the estimate \eqref{4.24} for $B_0(n,\a,\b)$. We proceed by induction. Suppose that the estimate
\eqref{4.24} is valid for  $i\leqslant k-1$ and let us establish it for $i=k$. Using the induction assumption,
the formulas \eqref{estAgamma}, \eqref{maincoeff}, and the positivity of the coefficients $\tau_i$, we obtain
\begin{equation}\label{prom1''}
\begin{aligned}
|B_k(n,\a,\b)|\leqslant & \sum\limits_{\gamma\in\mathds{Z}_+^2}
\sum\limits_{\substack{\varkappa\in\mathds{Z}_+^{\gamma_1} \\ |\varkappa|+\gamma_2=k}}
|A_\gamma(n,\a,\b)||B_{\varkappa_1}(n,\a,\b)|\cdot\dots\cdot |B_{\varkappa_{\gamma_1}}(n,\a,\b)|
\\
\leqslant & c_{8}\sum\limits_{\gamma\in\mathds{Z}_+^2}
\sum\limits_{\substack{\varkappa\in\mathds{Z}_+^{\gamma_1} \\ |\varkappa|+\gamma_2=k}}
c_{9}^{|\gamma|}\tau_{\varkappa_1}\cdot\dots\cdot\tau_{\varkappa_{\gamma_1}}.
\end{aligned}
\end{equation}
Then we make the same calculations with the function
$\widetilde{F}(\mu\widetilde{\varphi}(\mu), \mu)$ and we see easily that the coefficients at $\mu^k$ in its Taylor series is exactly the right hand side of \eqref{prom1''}. Since the function $\mu\widetilde{\varphi}(\mu)$ solves Equation~\eqref{equateta}
with $F$ replaced by $\widetilde{F}$, we conclude that the coefficient at $\mu^k$ in the Taylor series for $\widetilde{F}(\mu\widetilde{\varphi}(\mu), \mu)$ coincides with the coefficients at $\mu^k$ in the Taylor series for $\widetilde{\varphi}(\mu)$, that is, with $\tau_k$. Then we can replace the right hand side of \eqref{prom1''} by $\tau_k$ and this proves the estimate \eqref{4.24**}
for $B_k(n, \alpha, \beta)$.

Now we apply the Cauchy--Hadamard formula for the convergence radius of the series \eqref{varphiseries} and see that it is estimated from below by the constant $c_{10}$. The proof is complete.
\end{proof}

Lemma~\ref{lm4.3} shows that the unique solution $\eta=\varphi(\mu, n)$ of Equation~\eqref{equateta} can be represented by the series
\eqref{varphiseries} and this series converges for $|\mu|<c_{10}$, where $c_{10}$ independent of $n$, while the coefficients in \eqref{varphiseries} satisfy \eqref{maincoeff}.
The function $\eta\mapsto \arcsin\eta$ is holomorphic in $|\eta|<1$ and its Taylor series converges for these values of $\eta.$ Since the function $\eta(\mu)$ is   holomorphic and $\eta(0)=0$, the function $\xi(\mu)=\arcsin\eta(\mu)$ is also holomorphic and its Taylor series reads
\begin{equation}\label{xi111}
\xi(\mu)=\mu\sum\limits_{i=0}^\infty \rho_i(n,\alpha, \beta)\mu^i,
\end{equation}
where $\rho_i(n,\alpha, \beta)$ are some coefficients. This series surely converges provided $|\eta(\mu)|\leqslant \frac{1}{2}$ and in view of Lemma~\ref{lm4.3}, the latter inequality holds uniformly in $\a,$ $\b$ and $n$ provided
\begin{equation}\label{4.2}
|\mu|\leqslant c_{13},
\end{equation}
where $c_{13}$ is some fixed positive constant independent of $\a,$ $\b,$ $\mu$, and $n.$ It also follows from \eqref{varphiseries}, \eqref{4.24**}, and \eqref{xi111} that the coefficients $\rho_i$,
$i\geqslant 0$,   satisfy the estimates (\ref{4.3}).

Now we are going to find explicitly the  coefficients $\rho_i.$ Returning to Equation \eqref{eqforz}, we obtain
\begin{equation}\label{prom123}
(-1)^n\sin\xi(\mu)-\sum\limits_{j=0}^\infty\frac{\mu^{j+1}f_j
\big(\xi(\mu)+\pi n,\a,\b\big)}{(1+\xi\mu)^{j+1}}=0.
\end{equation}
Since all functions in the last equation are holomorphic, we can expand them into the Taylor series. We have
\begin{align*}
f_j\big(\xi(\mu)+\pi n,\a,\b\big) &=
\sum\limits_{s=0}^\infty\frac{\p_z^s f_j(\pi n,\a,\b)}{s!}\xi^s(\mu), \qquad
\sin\xi(\mu) =\xi(\mu)+\sum\limits_{s=1}^\infty\frac{(-1)^s}{(2s+1)!}\xi^{2s+1}(\mu), \\
\frac{\mu^{j+1}}{\big(1+\mu\xi(\mu)\big)^{j+1}} &= \sum\limits_{k=0}^\infty\frac{(-1)^k(k+j)!}{k!j!}\xi^k(\mu)\mu^{k+j+1}.
\end{align*}
Substituting these series into \eqref{prom123}, we obtain
$$
(-1)^n\Big(\xi(\mu)+\sum\limits_{s=1}^\infty\frac{(-1)^s}{(2s+1)!}\xi^{2s+1}(\mu)\Big)
-\sum\limits_{j=0}^\infty\sum\limits_{k=0}^\infty
\sum\limits_{s=0}^\infty\frac{(-1)^k(k+j)!\p_z^s f_j (\pi n,\a,\b)}{k!j!s!}\xi^{k+s}(\mu)\mu^{k+j+1}=0.
$$
In terms of new summation indices $t=k+s$ and $l=k+j+1$ this identity can be rewritten as
\begin{equation}\label{prom124}
(-1)^n\bigg(\xi(\mu)+\sum\limits_{s=1}^\infty\frac{(-1)^s}{(2s+1)!}\xi^{2s+1}(\mu)\bigg)
-\sum\limits_{l=1}^\infty\sum\limits_{t=0}^\infty E_{t, l}(n,\a,\b)\xi^{t}(\mu)\mu^{l}=0,
\end{equation}
where $E_{t, l}(n,\a,\b)$ is given by \eqref{Etl}. The identity \eqref{xi111} gives
$$
\xi^t(\mu)=\mu^t\Big(\sum\limits_{i=0}^\infty \rho_i(n,\alpha, \beta)\mu^i\Big)^t=
\sum\limits_{p=0}^\infty\mu^{t+p}\sum\limits_{\substack{\zeta\in\mathds{Z}_+^t \\ |\zeta|=p}}
\rho_{\zeta_1}(n,\a,\b)\cdot\ldots\cdot\rho_{\zeta_t}(n,\a,\b),
$$
where we also use (\ref{2.2}). Substituting this relation into \eqref{prom124}, we obtain
\begin{align*}
(-1)^n\bigg(\mu\sum\limits_{i=0}^\infty \rho_i(n,\alpha, \beta)\mu^i
&+\sum\limits_{s=1}^\infty\sum\limits_{m=0}^\infty H_{s, m}(n,\a,\b)\mu^{m+2s+1}\bigg) \\
&=\sum\limits_{l=1}^\infty\sum\limits_{t=0}^\infty\sum\limits_{p=0}^\infty E_{t, l}(n,\a,\b)\mu^{l+t+p}
\sum\limits_{\substack{\zeta\in\mathds{Z}_+^t \\ |\zeta|=p}}
\rho_{\zeta_1}(n,\a,\b)\cdot\ldots\cdot\rho_{\zeta_t}(n,\a,\b),
\end{align*}
where $H_{s, m}(n,\a,\b)$ is defined by \eqref{Hsm}.
Letting $i:=l+t+p-1$, we get
\begin{equation*}
(-1)^n\bigg(\sum\limits_{i=0}^\infty\rho_i(n,\alpha, \beta)\mu^i+\sum\limits_{s=1}^\infty\sum\limits_{m=0}^\infty
H_{s, m}(n,\a,\b)\mu^{m+2s}\bigg)=\sum\limits_{i=0}^\infty\omega_i(n,\alpha, \beta)\mu^i,
\end{equation*}
where $\omega_i(n,\alpha, \beta)$ are given by \eqref{Gp}. This identity imply
 the recurrent formula \eqref{polDi} for the coefficients $\rho_i.$

It remains to return to the original variables. Replacing $\mu$ by $(\pi n)^{-1},$ we conclude that the condition (\ref{4.2}) is satisfied provided
\begin{equation}\label{4.8}
 n\geqslant n_0, \qquad n_0:=\max\left\{\left[\frac{1}{\pi c_{13}}\right],\, \left[\frac{2c_{2}}{\pi}\right]\right\}+1
\end{equation}
and for such $n$ and all $\a,$ $\b$ we have the roots of Equation~(\ref{4.17'})
\begin{equation*}
z_n(\a,\b)=\pi n+\frac{1}{\pi n}\sum\limits_{i=0}^\infty\frac{\rho_i(n,\alpha, \beta)}{(\pi n)^i}.
\end{equation*}
Here the series  converges uniformly and absolutely in $\alpha$, $\beta$,
and large enough $n$.  Recovering then the eigenvalues of operator $\Op(\a,\b)$ by the formula $\l_n(\a,\b)=z_n^2(\a,\b),$ we arrive at~\eqref{prom11z},~\eqref{polDi}. Since by Lemma~\ref{lm4.2} the Cauchy  problem for the equation in \eqref{4.1}
has a unique solution, in view of the arguing before Equation~(\ref{4.12}) we conclude that the found eigenvalues  $\l_n(\a,\b)$ are simple. The estimate (\ref{4.19}) is immediately implied by
(\ref{4.3}), (\ref{4.8}).
The proof of Theorem~\ref{th1} is complete.

\section{Asymptotics for eigenvalues and basis property}\label{sec5}

In this section we prove  Theorems~\ref{th2},~\ref{th3}.

\begin{proof}[Proof of Theorem~\ref{th2}]
 It follows from  (\ref{4.19})  that
\begin{equation*}
\bigg|\sum\limits_{i=2}^{\infty} \frac{\rho_i(n,\alpha, \beta)}{(\pi n)^i}\bigg| \leqslant \frac{2c_{1}c_{2}^2}{\pi^2 n^2},\qquad \bigg|\sum\limits_{i=4}^{\infty} \frac{\rho_i(n,\alpha, \beta)}{(\pi n)^i}\bigg| \leqslant \frac{2c_{1}c_{2}^4}{\pi^4 n^4}\qquad\text{for}\quad n\geqslant n_0.
\end{equation*}
Then it follows from  (\ref{prom11z}) and (\ref{4.3}) that
\begin{align}\label{5.2}
&\l_n(\a,\b)=\pi^2 n^2 + 2\rho_0(n,\a,\b)+\frac{2\rho_1(n,\a,\b)}{\pi n} + \frac{2\rho_2(n,\a,\b)+\rho_0^2(n,\a,\b)}{\pi^2 n^2}+O(n^{-3}),
\\
&\l_n(\a,\b)=\pi^2 n^2 + 2\rho_0(n,\a,\b)+ O(n^{-1}),  \label{5.3}
\end{align}
as $n\to+\infty,$ where the estimates for the error terms are uniform in $\a$ and $\b.$

The formulas   \eqref{polDi} and \eqref{Gp} give
\begin{equation}\label{D0}
\begin{aligned}
\rho_0(n,\alpha, \beta) =&(-1)^nE_{0, 1}(n,\a,\b),
\\
\rho_1(n,\alpha, \beta) =&(-1)^n\big(E_{1, 1}(n,\a,\b)\rho_0(n,\alpha, \beta)
+E_{0, 2}(n,\a,\b)\big), \\
\rho_2(n,\a,\b) =& \frac{(\rho_0(n,\alpha, \beta))^3}{3!}
 +(-1)^n  E_{1, 1}(n,\a,\b)\rho_1(n,\alpha, \beta)
\\
&+(-1)^n E_{2, 1}(n,\a,\b) \rho_0^2(n,\alpha, \beta)
+(-1)^n E_{1, 2}(n,\a,\b)\rho_0(n,\alpha, \beta)
\\
&+(-1)^n E_{0, 3}(n,\a,\b).
\end{aligned}
\end{equation}
Using the definition \eqref{Etl}, we obtain
\begin{align*}
&E_{0, 1}(n,\a,\b)=f_0(\pi n,\a,\b), &&
E_{0, 2}(n,\a,\b)= f_1(\pi n,\a,\b),
\\
&E_{1, 1}(n,\a,\b)=\partial_z f_0(\pi n,\a,\b),  &&
E_{2, 1}(n,\a,\b)=\frac{\partial_z^2f_0(\pi n,\alpha,\beta)}{2},
\\
&E_{1, 2}(n,\a,\b)= -f_0(\pi n, \alpha, \beta)
+ \partial_z f_1(\pi n,\a,\b), && E_{0, 3}(\pi n,\a,\b)=f_2(\pi n,\a,\b).
\end{align*}
Substituting these formulas into (\ref{D0}), we obtain
  the formulas \eqref{coeffD}.

Our next step is the formulas for  the functions $f_j$. We introduce the
functions for $z\in\mathds{C}$, $s\in [0, 1]$:
\begin{align*}
\Xi_0(z, s) = &V(s)\big(\mathcal{R}\cos z(\,\cdot\,+\alpha)\big)(z,s) +Q(s)\big(\mathcal{R}\cos z(\,\cdot\,-\beta)\big)(z,s), \\
\Xi_1(z, s) =& V(s)\mathcal{R}\int\limits_0^{s+\alpha}\sin \pi n(s+\alpha-\tau)\Xi_0(z, \tau)\,d\tau \\
& +Q(s)\mathcal{R}\int\limits_0^{s-\beta}\sin \pi n(s-\beta-\tau)\Xi_0(z, \tau)\,d\tau.
\end{align*}
The formulas \eqref{B}, \eqref{OperM}, and \eqref{fj} imply
\begin{equation}
\label{polf0}
\begin{aligned}
f_0(\pi n,\a,\b) &= (-1)^n\int\limits_0^{1-\alpha}V(t)\cos\pi n(t+\alpha)\cos\pi nt\,dt
+(-1)^n\int\limits_\beta^1Q(t)\cos\pi n(t-\beta)\cos\pi nt\,dt,
\end{aligned}
\end{equation}
and hence,
\begin{equation} \label{polf0'}
\begin{aligned}
\partial_z f_0(\pi n,\a,\b)=& (-1)^{n+1}\int\limits_0^{1-\alpha}V(t)\big((t+\alpha)\sin\pi n(t+\alpha)\cos\pi nt
\\
&\hphantom{(-1)^{n+1}\int\limits_0^{1-\alpha}V(t)\big(}-(1-t)\cos\pi n(t+\alpha)\sin\pi nt\big)\,dt \\
& +(-1)^{n+1}\int\limits_\beta^1Q(t)\big((t-\beta)\sin\pi n(t-\beta)\cos\pi nt
\\
&\hphantom{+(-1)^{n+1}\int\limits_\beta^1Q(t)\big(}-(1-t)\cos\pi n(t-\beta)\sin\pi nt\big)\,dt,
\\
 \partial_z^2f_0(\pi n,\alpha,\beta)=&(-1)^{n+1}\int\limits_0^{1-\alpha}V(t)\big((t+\alpha)^2
+(1-t)^2\big)\cos\pi n(t+\alpha)\cos\pi nt\,dt \\
& +2(-1)^{n+1}\int\limits_0^{1-\alpha}V(t)(t+\alpha)(1-t)\sin\pi n(t+\alpha)\sin\pi nt\,dt \\
& +(-1)^{n+1}\int\limits_\beta^1 Q(t)\big((t-\beta)^2+(1-t)^2\big)\cos\pi n(t-\beta)\cos\pi nt\,dt
\\
& +2(-1)^{n+1}\int\limits_\beta^1 Q(t)(t-\beta)(1-t)\sin\pi n(t-\beta)\sin\pi nt\,dt.
\end{aligned}
\end{equation}
In the same way we obtain
\begin{equation}\label{polf1}
\begin{aligned}
f_1(\pi n,\a,\b) = &(-1)^n\int\limits_0^1\Bigg(V(t)\mathcal{R}\bigg(\int\limits_0^{\,\cdot\,+\alpha}
\Xi_0(\pi n, s)\sin \pi n(\,\cdot\,+\alpha-s)\,ds\bigg)(t,\a) \\
& +Q(t)\mathcal{R} \bigg(\int\limits_0^{\,\cdot\,-\beta}\Xi_0(\pi n, s)\sin \pi n(\,\cdot\,-\beta-s)\,ds\bigg)(t,\beta)\Bigg)\cos \pi nt\,dt,\\
 \partial_z f_1(\pi n,\a,\b)= &(-1)^n\int\limits_0^1\Bigg(V(t)\mathcal{R}\bigg(\int\limits_0^{\,\cdot\,+\alpha}
\Xi_0(\pi n, s)\sin \pi n(\,\cdot\,+\alpha-s)\,ds\bigg)(t,\a) \\
& +Q(t)\mathcal{R}\bigg(\int\limits_0^{\,\cdot\,-\beta}\Xi_0(\pi n, s)\sin \pi n(\,\cdot\,-\beta-s)\,ds\bigg)(t,\b)\Bigg)(1-t)\sin\pi nt\,dt \\
& +(-1)^n\int\limits_0^1\Bigg(V(t)\mathcal{R} \bigg(\int\limits_0^{\,\cdot\,+\alpha}\Big(-V(s)(s+\alpha)\mathcal{R}\big(\sin \pi n(\,\cdot\,+\alpha)\big)(s,\a) \\
& -Q(s)(s-\beta)\mathcal{R}\big(\sin \pi n(\,\cdot\,-\beta)\big)(s,\beta)\Big)\sin \pi n(t+\alpha-s)\,ds\bigg)(t,\a,\b) \\
& +V(t)\mathcal{R}\bigg(\int\limits_0^{\,\cdot\,+\alpha}\Xi_0(\pi n, s)(\,\cdot\,+\alpha-s)\cos \pi n(\,\cdot\,+\alpha-s)\,ds\bigg)(t,\a) \\
& +Q(t)\mathcal{R}\bigg(\int\limits_0^{\,\cdot\,-\beta} \Big(-V(s)(s+\alpha)\mathcal{R}\big(\sin \pi n(\,\cdot\,+\alpha)\big)(s,\a) \\
& -Q(s)(s-\beta)\mathcal{R}\big(\sin \pi n(\,\cdot\,-\beta)\big)(s,\b)\Big)\sin \pi n(t-\beta-s)\,ds\bigg)(t,\a,\b) \\
& +Q(t)\mathcal{R}\bigg(\int\limits_0^{\,\cdot\,-\beta}\Xi_0(\pi n, s)(\,\cdot\,-\beta-s)\cos \pi n(\,\cdot\,-\beta-s)\,ds\bigg)\Bigg)\cos \pi nt\,dt,
\end{aligned}
\end{equation}
and, finally,
\begin{equation}\label{polf2}
\begin{aligned}
f_2(\pi n,\a,\b)
=(-1)^n\int\limits_0^1\Bigg(&V(t)\mathcal{R}\bigg(\int\limits_0^{\,\cdot\,+\alpha} \Xi_1(\pi n, s)\sin \pi n(\,\cdot\,+\alpha-s)\,ds\bigg)(t,\a)
\\
& +Q(t)\mathcal{R}\bigg(\int\limits_0^{\,\cdot\,-\beta}\Xi_1(\pi n, s)\sin \pi n(\,\cdot\,-\beta-s)\,ds\bigg)\Bigg)\cos \pi nt\,dt.
\end{aligned}
\end{equation}
Substituting \eqref{polf0} into \eqref{5.3}, we obtain \eqref{mainasymptinfty}.

Let $V$ and $Q$ belongs to $W_\infty^2(0, 1)$.
Then we can integrate by parts in
\eqref{polf0}, (\ref{polf0'}), (\ref{polf1}), \eqref{polf2}. We have
\begin{align*}
f_0(\pi n,\a,\b)= \frac{(-1)^n}{2}\bigg(&G_{n, 1}(\alpha, \beta)
+\frac{G_{n, 4}(\alpha, \beta)}{2\pi n}
\\
&+\frac{\big(V'(1-\alpha)-V'(0)\big)\cos\pi n\alpha}{4\pi^2n^2}
+\frac{\big(Q'(1)-Q'(\beta)\big)\cos\pi n\beta}{4\pi^2 n^2}\\
&-\frac{1}{4\pi^2n^2}\int\limits_0^{1-\alpha}V''(t)\cos\pi n(2t+\alpha)\,dt
-\frac{1}{4\pi^2n^2}\int\limits_\beta^{1}Q''(t)\cos\pi n(2t-\beta)\,dt\bigg),
\end{align*}
and
\begin{align*}
\partial_z f_0(\pi n,\a,\b) &= \frac{(-1)^{n+1}}{2}\bigg(G_{n, 2}(\alpha, \beta)
+\frac{G_{n, 5}(\alpha, \beta)}{2\pi n}\bigg) +O(n^{-2}), \\
 \partial_z^2 f_0(\pi n,\alpha,\beta) &=\frac{(-1)^{n+1}G_{n, 3}(\alpha, \beta)}{2}+O(n^{-1}),
  \\
f_1(\pi n,\a,\b) &=\frac{(-1)^nG_{n, 6}(\alpha, \beta)}{2\pi n}+O(n^{-2}),
\\
 \partial_z f_1(\pi n,\a,\b)&=O(n^{-1}),\qquad f_2(\pi n,\a,\b)=O(n^{-1}),
\end{align*}
where $G_{n, j}$, $j=1, 2, 3, 4, 5, 6$, were defined in \eqref{Gnj}.
Substituting these asymptotics into \eqref{coeffD} and the result into \eqref{5.2}, we obtain
the formula \eqref{lambdan1sm}. The proof is complete.
\end{proof}

\begin{proof}[Proof of Theorem~\ref{th3}]
It follows from \cite[Theorem~3.4, Remark~3.5]{MityaginSiegl} for $\alpha=0$ and $\gamma=2$ that the system
of eigenfunctions and generalized eigenfunctions forms of the operator $\mathcal{H}(\alpha, \beta)$ forms
the Bari basis in the space $L^2(0, 1)$. The proof is complete.
\end{proof}

\section*{Declarations}

\textbf{Funding}

This work is supported by the Russian Science
Foundation, grant no. 23-11-00009,\\ https: //rscf.ru/project/23-11-00009/

\textbf{Ethical Approval}

Not applicable

\textbf{Competing interests}

The authors have no competing interests to declare that are relevant to the
content of this article.

\textbf{Authors' contributions}

Not applicable

\textbf{Availability of data and materials}

Not applicable

\end{document}